\documentclass{amsart}
\usepackage{enumerate}
\usepackage{graphicx}
\usepackage{amsaddr}
\usepackage{amsthm}
\usepackage{subfigure}
\newtheorem{thm}{Theorem}[section]

\newtheorem{lem}[thm]{Lemma}
\newtheorem{prop}[thm]{Proposition}
\theoremstyle{definition}
\newtheorem{defn}[thm]{Definition}

\newtheorem{notn}[thm]{Notation}
\theoremstyle{remark}
\newtheorem{rem}[thm]{Remark}
\newtheorem{claim}{Claim}
\numberwithin{equation}{section}

\newcommand{\I}{\mathcal{I}}
\newcommand{\Text}[1]{\text{\textnormal{#1}}}

\newcommand{\cl}{\textrm{cl}}
\begin{document}

\title{on circuits and serial symmetric basis-exchange in matroids}%
\author{Daniel Kotlar}%
\address{Computer Science Department, Tel-Hai College, Upper Galilee 12210, Israel}%
\email{dannykot@telhai.ac.il}%

\thanks{The author thanks an anonymous referee for many insightful comments and for offering shorter and more transparent proofs, than the ones originally provided by the author, for Proposition~\ref{prop6:1} and Theorem 5.3.}%

\begin{abstract}
The way circuits, relative to a basis, are affected as a result of exchanging a basis element, is studied. As consequences, it is shown that three consecutive symmetric exchanges exist for any two bases of a matroid, and that a full serial symmetric exchange, of length at most 6, exists for any two bases of a matroid of rank 5. A new characterization of binary matroids, related to basis-exchange, is presented.
\end{abstract}
\maketitle
\section{Introduction}
A matroid $M$ consists of a hereditary family $\I$ of subsets (called \emph{independent sets}) of a finite ground set $E$ that satisfies an augmentation axiom: If $A,B \in \I$ and $|B|>|A|$ then there exists $x\in B\setminus A$ such that $A+x\in \I$ (the common notation of $A+x$ for $A\cup\{x\}$ and $A-x$ for $A\setminus\{x\}$ is used here). A maximal independent set is called a \emph{basis}. An element $x\in E$ is \emph{spanned} by $A$ if either $x\in A$ or $I+x\not\in \I$ for some independent set $I\subseteq A$. The \emph{rank} of $A\subseteq E$, denoted $r(A)$, is the size of a maximal independent subset in $A$. A \emph{circuit} is a minimal dependent set. When $I$ is independent but $I+x$ is not, the unique circuit in $I+x$ will be denoted by $C(x,I)$. For a matroid $M$, its dual matroid $M^*$ is the matroid on the same ground set whose bases are the complements of the bases of $M$. The circuits of $M^*$ are called \emph{cocircuits} of $M$. For further information about matroid theory the reader is referred to Oxley~\cite{Oxley11} and Welsh~\cite{Welsh76}.

The notion of \emph{basis-exchange} refers to the replacement of one element (or a set of elements) in a basis by an element (resp. a set of elements) outside the basis so that the resulting new set is also a basis, or to the exchange of one or more elements between two bases so that the resulting two new sets are also bases (\emph{symmetric basis-exchange}). For some basic results regarding basis-exchange the reader is referred to Greene~\cite{Greene71,Greene73} and Kung~\cite{Kung86}. Greene~\cite{Greene71} used basis-exchange properties to characterize binary matroids (see also \cite{Bonin98}).

A \emph{serial basis-exchange} is a basis-exchange that is performed one element at a time so that after each exchange the resulting set is a basis. When symmetrically exchanging one element at a time between two bases the process is termed \emph{serial symmetric basis-exchange}. A \emph{full serial symmetric basis-exchange} is a serial symmetric basis-exchange that ends when all the elements were exchanged. An interesting, still unsolved, problem, first presented by Gabow~\cite{Gabow76}, is the following:
\\
\\
\textbf{Basis-Exchange Problem:} Does a full serial symmetric basis-exchange exist for any two bases of a matroid?
\\

This problem was presented in a more general form by Kajitani and Sugishita~\cite{Kaj83} and further possible generalizations appear in the work of van den Heuvel and Thomass\'{e}~\cite{HT}. It was proved for graphic block matroids by Farber et al. \cite{FRS85} (with a modification by Weidemann \cite{Wei}), and independently by  Kajitani et al. \cite{Kaj88} and Cordovil and Moreira \cite{Cor93}. It was also proved for transversal block matroids by Farber \cite{Far89}. Recently Bonin~\cite{Bonin10} proved it for sparse paving matroids and Kotlar and Ziv~\cite{KZ2011} proved it for any matroid of rank 4, while showing that any pair of elements in one basis can be symmetrically and serially exchanged with some pair of elements in the other basis.

In the works mentioned above it is assumed that in a full serial symmetric basis-exchange each element is exchanged exactly once so that the process involves $r(M)$ steps. However, referring to the \emph{length} of a serial symmetric basis-exchange as the number of one-element exchanges performed, we can ask a weaker question, namely, \emph{whether a full serial symmetric exchange of any length exists for any two bases}, or \emph{what is the minimal length of a full serial symmetric basis-exchange that we can prove to exist}? If we refer to a serial symmetric exchange that is not full as \emph{partial serial symmetric exchange}, another question, not considered so far, is \emph{what is the minimal length of a partial serial symmetric exchange that we can ensure to exist between any two bases?}
\\

Before concluding this introduction we describe some more terminology and notation:
\\
Let $A$ and $B$ be two disjoint bases of a matroid $M$, and let $a\in A$ and $b\in B$. We say that $a$ and $b$ are \emph{exchangeable relative to} $A$ \emph{and} $B$ if $A-a+b$ and $B-b+a$ are bases. When $a$ and $b$ are exchangeable relative to two known bases, the exchange of $a$ and $b$ will be denoted by $a\leftrightarrow b$. The notation $Sym(a,A,B)$ will be used to describe the set of elements in $B$ that are exchangeable with $a$, relative to $A$ and $B$. It will be convenient to use the directed bipartite graph whose parts are $A$ and $B$, and an edge from $a\in A$ to $b\in B$ indicates that $B-b+a$ is a basis. Thus, a two-sided edge between $a\in A$ and $b\in B$ indicates that $a$ and $b$ are exchangeable relative to $A$ and $B$. When the bases involved are known we shall sometimes refer to the relation $a\rightarrow b$ to indicate that in the corresponding bipartite graph there is an edge from $a$ to $b$ (and similarly for $a\leftarrow b$).
\\

Section~\ref{sec2} describes how the circuits $C(a,B)$ evolve as a result of exchanging one element in the basis $B$ and how the sets $Sym(a,A,B)$ evolve as a result of symmetrically exchanging one element between the bases $A$ and $B$. These results are used in Section~\ref{sec3} to show that a partial serial symmetric basis-exchange of length 3 exists for any two bases of a matroid. In section~\ref{sec4} it is shown that a full serial symmetric exchange of length at most 6 exists for any two bases of a matroid of rank 5. Section~\ref{sec5} contains related results for binary matroids, including a basis-exchange characterization of binary matroids.
\section{The effect of exchanging one basis element}\label{sec2}
For any two sets $X$ and $Y$ let $X\bigtriangleup Y$ be the symmetric difference $(X\setminus Y)\cup (Y\setminus X)$. The proof of the next proposition uses the following axiom (see \cite{Oxley11}).
\\
\\
\textbf{Strong Circuit Elimination Axiom:} Let $C_{1}$ and $C_{2}$ be circuits and let $x\in C_{1}\cap C_{2}$ and $y\in C_{1}\bigtriangleup C_{2}$. Then there exists a circuit $C\subset C_{1}\cup C_{2}$ such that $x\not\in C$ and $y\in C$.
\\
\begin{notn}\label{notn1}
Let $B$ be a basis in a matroid $M$ and let $x\not\in B$. For $b\in C(x,B)\cap B$ let $B_{b,x}$ denote the basis $B-b+x$.
\end{notn}
Now, let $y\not\in B$, $y\neq x$. If $b\not\in C(y,B)$ then $C(y,B_{b,x})=C(y,B)$. If $b\in C(y,B)$, the following proposition provides information on the circuit $C(y,B_{b,x})$.
\begin{prop}\label{prop2:1}
Let $B$ be a basis in a matroid $M$. Let $x,y\not\in B$ ($x\ne y$), and $b\in C(x,B)\cap C(y,B)$, then
\begin{equation}\label{eq1}
C(x,B)\bigtriangleup C(y,B)\subseteq C(y,B_{b,x})\subseteq C(x,B)\cup C(y,B).
\end{equation}
\end{prop}
\begin{proof}
Suppose $b'\in B$ and $b'\in C(y,B)\setminus C(x,B)$. By the Strong Circuit Elimination Axiom $C(x,B)\cup C(y,B)$ contains a circuit $C$ that contains $b'$ but not $b$. The circuit $C$ must contain at least one of $x$ and $y$ (otherwise $C\subset B$, which is ruled out since $B$ is independent). If $C$ contains just $x$ (or just $y$) then $C=C(x,B)$ (or $C=C(y,B)$), contrary to the assumption that $b\in C(x,B)$ (respectively, $b\in C(y,B)$). Thus $x,y\in C$ and hence $C=C(y, B_{b,x})$. It follows that $b'\in C(y, B_{b,x})$ and $C(y, B_{b,x})\subseteq C(x,B)\cup C(y,B)$. Thus $C(y,B)\setminus C(x,B)\subseteq C(y, B_{b,x})$. The proof that $C(x,B)\setminus C(y,B)\subseteq C(y, B_{b,x})$ is similar.
\end{proof}

A matroid $M$ whose ground set is the disjoint union of two bases is called a \emph{block matroid}. Let $A$ and $B$ be two disjoint bases in a block matroid $M$ so that $E=A\cup B$. Since $A=E\setminus B$, both $A$ and $B$ are cobases. For $a\in A$ let $C^*(a,B)$ be the unique cocircuit in $B+a$, where $B$ is viewed as a cobase.

For two disjoint bases $A$ and $B$ and any two distinct elements $a,a'\in A$, the following set was defined in \cite{KZ2011}:
\begin{equation*}
Conn(a,a',A,B)=\{b\in B|A-a'+b\;\Text{and}\;B-b+a\;\Text{are bases}\}
\end{equation*}
The rationale for this notation comes from the corresponding directed bipartite graph, in which $Conn_{B}(a,a')$ is the set of elements of $B$ that ``connect'' $a$ to $a'$ through a directed path of length two. Note that
\begin{equation}\label{eq1:1}
Sym(a,A,B)\cap Sym(a',A,B)=Conn(a,a',A,B)\cap Conn(a',a,A,B).
\end{equation}
By restricting to a block matroid and using the known fact that in a block matroid $C^*(a,B)=\{b\in B | a\in C(b,A)\}+a$ we obtain:
\begin{prop}\label{prop3:1}
If $A$ and $B$ are disjoint bases in a block matroid $M$ with $a,a'\in A$ and $a\neq a'$, then
\begin{enumerate}
\item[\Text{(i)}] $Sym(a,A,B)= \left(C(a,B)\cap C^*(a,B)\right)-a$, and
\item[\Text{(ii)}] $Conn(a,a',A,B)= C(a,B)\cap C^*(a',B)$.
\end{enumerate}
\end{prop}
From Part (ii) of Proposition~\ref{prop3:1} it immediately follows that
\begin{prop}\label{prop:KZ}
For any two bases $A$ and $B$ of a matroid $M$, and any two distinct elements $a,a'\in A$, $|Conn(a,a',A,B)|\neq1$.
\end{prop}

Now, an analogue of Proposition~\ref{prop2:1} can be obtained in the case of a symmetric exchange.
\begin{prop}\label{prop3:4}
Let $A$ and $B$ be disjoint bases in a matroid $M$ and let $b\in Sym(a,A,B)$. Then, for any $a'\in A-a$
\begin{enumerate}
\item[\Text{(i)}] If $b\not\in C(a',B)$ and $a'\not\in C(b,A)$ then
\begin{equation*}
Sym(a',A_{a,b},B_{b,a}) = Sym(a',A,B)
\end{equation*}
\item[\Text{(ii)}] If $b\in C(a',B)$ and $a'\not\in C(b,A)$ then
\begin{equation*}
\begin{split}
Sym(a',A,B)\bigtriangleup Conn(a,a',A,B)&\subseteq Sym(a',A_{a,b},B_{b,a})\\
&\subseteq Sym(a',A,B)\cup Conn(a,a',A,B)
\end{split}
\end{equation*}
\item[\Text{(iii)}] If $b\not\in C(a',B)$ and $a'\in C(b,A)$ then
\begin{equation*}
\begin{split}
Sym(a',A,B)\bigtriangleup Conn(a',a,A,B)&\subseteq Sym(a',A_{a,b},B_{b,a})\\
&\subseteq Sym(a',A,B)\cup Conn(a',a,A,B)
\end{split}
\end{equation*}
\item[\Text{(iv)}] If $b\in C(a',B)$ and $a'\in C(b,A)$ then
\begin{equation*}
\begin{split}
&Sym(a',A,B)\bigtriangleup Sym(a,A,B)\bigtriangleup Conn(a',a,A,B)\bigtriangleup Conn(a,a',A,B)+a \\
&\subseteq Sym(a',A_{a,b},B_{b,a})\\
&\subseteq Sym(a',A,B)\cup Sym(a,A,B)\cup Conn(a',a,A,B)\cup Conn(a,a',A,B)+a
\end{split}
\end{equation*}
\end{enumerate}
\end{prop}
\begin{proof}
By restricting to $A\cup B$ we may assume that $M$ is a block matroid. If $b\not\in C(a',B)$, then $C(a',B)=C(a',B_{b,a})$, otherwise $C(a',B)\triangle C(a,B)\subseteq C(a',B_{b,a})\subseteq C(a',B)\cup C(a,B)$ by Proposition~\ref{prop2:1}. Likewise, if $a'\not\in C(b,A)$ (equivalently, $b\not\in C^*(a',B)$), then $C^*(a',B) = C^*(a',B_{b,a})$, otherwise $C^*(a',B)\triangle C^*(a,B)\subseteq C^*(a',B_{b,a})\subseteq C^*(a',B)\cup C^*(a,B)$. By Proposition~\ref{prop3:1}, we have $Sym(a',A_{a,b},B_{b,a}) = (C(a',B_{b,a})\cap C^*(a',B_{b,a}))- a'$. All cases now follow easily from Proposition~\ref{prop3:1} along with the following equalities:
\begin{enumerate}
\item $(W\triangle X)\cap Y = (W\triangle Y)\cap ( X\triangle Y)$
and
\item $(W\triangle X)\cap (Y\triangle Z) = (W\cap Y)\triangle (W\cap
Z)\triangle (X\cap Y)\triangle (X\cap Z).$
\end{enumerate}
\end{proof}
\section{Three consecutive symmetric exchanges}\label{sec3}
In this section Proposition~\ref{prop3:4} is used to prove that three consecutive symmetric exchanges are always possible between any two bases of a matroid. Recall Notation~\ref{notn1} which will be frequently used here.
\begin{lem}\label{lem3:2}
Let $A$ and $B$ be two disjoint bases of a matroid $M$. Suppose $a_1,a_2\in A$ are serially and symmetrically exchangeable with $b_1,b_2\in B$ (in this order). Let $A_1=A_{a_1,b_1}$, $B_1=B_{b_1,a_1}$, $A_2=(A_1)_{a_2,b_2}$ and $B_2=(B_1)_{b_2,a_2}$. If $b_1,b_2\not\in Sym(a_1,B_2,A_2)$, then the only possible relations among the elements $a_1,a_2,b_1$ and $b_2$, relative to $A_2$ and $B_2$, are the ones described in Figure~\ref{fig1}.
\begin{figure}[h!]
  \centering
  \subfigure[]{\label{fig1a}\includegraphics[scale=0.35]{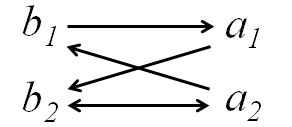}}
  \subfigure[]{\label{fig1b}\includegraphics[scale=0.35]{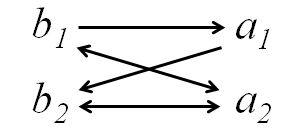}}
  \subfigure[]{\label{fig1c}\includegraphics[scale=0.35]{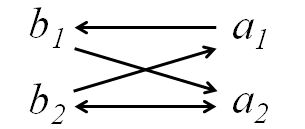}}
  \subfigure[]{\label{fig1d}\includegraphics[scale=0.35]{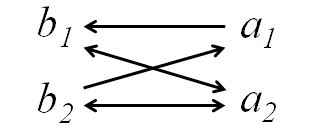}}
  \caption{Possible relations between the elements $a_1,a_2\in B_2$ and $b_1,b_2\in A_2$, given that $b_1,b_2\not\in Sym(a_1,B_2,A_2)$.}
  \label{fig1}
\end{figure}
\end{lem}
\begin{proof}
We apply Proposition~\ref{prop3:4} to the exchange of $b_2$ with $a_2$ relative to $A_2$ and $B_2$ (that is, exchanging them back). We have $A_1=A_2-b_2+a_2$ and $B_1=B_2-a_2+b_2$, and we look at $Sym(a_1,B_1,A_1)$. Since $b_1\not\in Sym(a_1,B_2,A_2)$ and $b_1\in Sym(a_1,B_1,A_1)$ this is not case (i) of Proposition~\ref{prop3:4}, and since $b_2\not\in Sym(a_1,B_2,A_2)$ this is not case (iv) either. Suppose Proposition~\ref{prop3:4}(ii) applies. That is, $b_2\in C(a_1,A_2)$ and $a_1\not\in C(b_2,B_2)$. By Proposition~\ref{prop3:4}(ii), $Sym(a_1,B_2,A_2)\bigtriangleup Conn(a_2,a_1,B_2,A_2)\subseteq Sym(a_1,B_1,A_1)\subseteq Sym(a_1,B_2,A_2)\cup Conn(a_2,a_1,B_2,A_2)$. Since $b_1\not\in Sym(a_1,B_2,A_2)$ and $b_1\in Sym(a_1,B_1,A_1)$ we must have that $b_1\in Conn(a_2,a_1,B_2,A_2)$. This implies either Figure~\ref{fig1a} or Figure~\ref{fig1b}. If Proposition~\ref{prop3:4}(iii) applies, then $b_2\not\in C(a_1,B_2)$ and $a_1\in C(b_2,A_2)$. By Proposition~\ref{prop3:4}(iii), $Sym(a_1,B_2,A_2)\bigtriangleup Conn(a_1,a_2,B_2,A_2)\subseteq Sym(a_1,B_1,A_1)\subseteq Sym(a_1,B_2,A_2)\cup Conn(a_1,a_2,B_2,A_2)$. Since $b_1\not\in Sym(a_1,B_2,A_2)$ we must have that $b_1\in Conn(a_1,a_2,B_2,A_2)$. This implies either Figure~\ref{fig1c} or Figure~\ref{fig1d}.
\end{proof}
In order to prove the main result we shall need two more definitions.
\begin{defn}
Let $A$ and $B$ be bases of a matroid $M$, and let $A'$ and $B'$ be bases obtained from $A$ and $B$ by sequentially exchanging elements, starting from $A$ and $B$. The \emph{distance} $d((A,B),(A',B'))$ between the two pairs of bases is defined as $|A\setminus A'|$.
\end{defn}
For example, the bases $A_2$ and $B_2$ in Figure~\ref{fig1_5} were obtained from $A$ and $B$ by sequentially exchanging $a_1$ with $b_1$ and $a_2$ with $b_2$. Thus, $d((A,B),(A_2,B_2))=2$.

\begin{figure}[h!]
  \centering
  \subfigure[]{\label{fig1_5a}\includegraphics[scale=0.35]{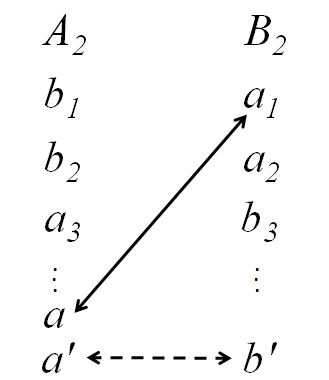}}
  \subfigure[]{\label{fig1_5b}\includegraphics[scale=0.35]{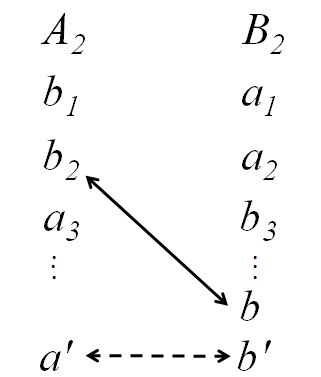}}
  \caption{Examples of correcting exchanges. The solid arrow represents the correcting exchange. The dashed arrow represents the increasing exchange.}
  \label{fig1_5}
\end{figure}

\begin{defn}
Let $A$ and $B$ be bases of a matroid $M$ and let $A'$ and $B'$ be bases obtained from $A$ and $B$ by sequentially exchanging elements, starting from $A$ and $B$. A \emph{correcting exchange} between $A'$ and $B'$ is an exchange, relative to $A'$ and $B'$, which preserves the distance from $(A,B)$, after which an additional exchange, which increases the distance from $(A,B)$, is enabled. This last additional exchange will be refered to as the \emph{increasing exchange}.
\end{defn}
In other words, a correcting exchange ``exchanges back'' an element in order to obtain a new exchange that increases the distance from $(A,B)$.

Figure~\ref{fig1_5} illustrates two examples of correcting exchanges. The solid arrow represents the correcting exchange, which preserves the distance 2 from $(A,B)$. The dashed arrow represents the increasing exchange, after which the distance from $(A,B)$ increases to 3.
\begin{thm}\label{thm1}
If $A$ and $B$ are disjoint bases of a matroid $M$, then there exist three elements in $A$ that can be serially and symmetrically exchanged with some three elements in $B$.
\end{thm}
\begin{proof}
The outline of the proof is as follows. We start with a serial symmetric exchange of length two (which exists by \cite{KZ2011}) between some $a_1,a_2\in A$ and $b_1,b_2\in B$, obtaining two bases $A_2$ and $B_2$, respectively. Assuming there is no third exchange between an element of $A$ and an element of $B$ it is shown that there exists a correcting exchange between $A_2$ and $B_2$. The key step of the proof is to show that the two bases obtained after the correcting exchange (and before performing the increasing exchange) can be obtained from $A$ and $B$ by performing only two exchanges. These two exchanges, together with the increasing exchange imply the claim of the theorem.

Since we know from \cite{KZ2011} that when $r(M)\leq 4$ there exists a full serial symmetric exchange, we may assume that $r(M)\geq 5$. Also, by \cite{KZ2011}, there exists a serial symmetric exchange of length 2. Let $a_1,a_2\in A$ and $b_1,b_2\in B$ be such that $A_1=A-a_1+b_1$, $B_1=B-b_1+a_1$, $A_2=A_1-a_2+b_2$ and $B_2=B_1-b_2+a_2$ are all bases, and we assume that there is no third exchange.

The first two claims imply that a correcting exchange exists.
\begin{claim}\label{claim:1}
There exist $a,a'\in A\cap A_2$, $b\in B\cap B_2$ and $i\in\{1,2\}$ such that $a_i\in Sym(a,A_2,B_2)\cap Sym(a',A_2,B_2)$ and $b\in Conn(a,a',A_2,B_2)$ (Figure~\ref{fig2}).
\end{claim}
\begin{figure}[h!]
\begin{center}
\includegraphics[scale=0.35]{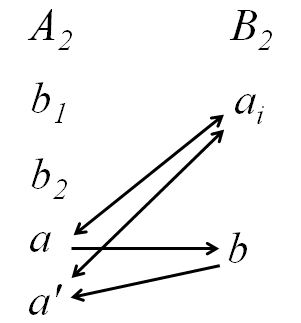}
\end{center}
\caption{Existence of  $a,a'$ and $b$ such that $a_i\in Sym(a,A_2,B_2)\cap Sym(a',A_2,B_2)$ and $b\in Conn(a,a',A_2,B_2)$ (Claim 1).}
\label{fig2}
\end{figure}
\begin{proof}[Proof of Claim~\ref{claim:1}]
\renewcommand{\qedsymbol}{}
Since each element of $A_2$ is symmetrically exchangeable with some element of $B_2$, the assumption that there is no third exchange implies that every element of $A\cap A_2$ must be symmetrically exchangeable with either $a_1$ or $a_2$. Since $r(M)\geq5$ there must be two elements $a,a'\in A\cap A_2$ that are symmetrically exchangeable with the same element of $\{a_1,a_2\}$. We assume it is $a_1$ (the reader can check that choosing $a_2$ yields the same conclusion). By (\ref{eq1:1}) $a_1\in Conn(a,a',A_2,B_2)\cap Conn(a',a,A_2,B_2)$. By Proposition~\ref{prop:KZ} $Conn(a,a',A_2,B_2)$ must contain at least one other element $x\in B_2$ and $Conn(a',a,A_2,B_2)$ must contain at least one other element $x'\in B_2$. There are two options for $x$ (and for $x'$): $x\in B\cap B_2$ and $x=a_2$. If $x\in B\cap B_2$ or $x'\in B\cap B_2$, then $b=x$ or $b=x'$ and we are done (in the case that $b=x'$ we exchange the roles of $a$ and $a'$). Hence, we assume that $x=x'=a_2$. Then $a_1,a_2\in Sym(a,A_2,B_2)$. Now, $C(a,B_2)$ must contain at least one element $b\in B_2\cap B$ (otherwise $C(a,B_2)$ would contain only elements of $A$). Similarly, $C(b,A_2)$ must contain an element $a''\in A_2\cap A$. Thus, $b\in Conn(a,a'',A_2,B_2)\cap B$. Since both $a_1$ and $a_2$ are in $Sym(a,A_2,B_2)$ and at least one of them is in $Sym(a'',A_2,B_2)$ (since we assume there is no third exchange), Claim~\ref{claim:1} follows with $a''$ instead of $a'$.
\end{proof}
\begin{claim}\label{claim:2}
Keeping the notation of Claim~\ref{claim:1}, let $A_3=(A_2)_{a,a_i}$, $B_3=(B_2)_{a_i,a}$, $A_3'=(A_2)_{a',a_i}$ and $B_3'=(B_2)_{a_i,a'}$. Then $b\in Sym(a',A_3,B_3)$ and $b\in Sym(a,A_3',B_3')$. In other words, the exchanges $a\leftrightarrow a_i$ and $a'\leftrightarrow a_i$ in Figure~\ref{fig2} are correcting exchanges.
\end{claim}
\begin{proof}[Proof of Claim~\ref{claim:2}]
\renewcommand{\qedsymbol}{}
By Proposition~\ref{prop3:4}(iv), $Sym(a',A_2,B_2)\bigtriangleup Sym(a,A_2,B_2)\bigtriangleup Conn(a',a,A_2,B_2)\bigtriangleup Conn(a,a',A_2,B_2)+a\subseteq Sym(a',A_3,B_3)$ and $Sym(a',A_2,B_2)\bigtriangleup Sym(a,A_2,B_2)\bigtriangleup Conn(a',a,A_2,B_2)\bigtriangleup Conn(a,a',A_2,B_2)+a'\subseteq Sym(a,A_3',B_3')$. Since $b$ is not in any of the sets $Sym(a,A_2,B_2)$, $Sym(a',A_2,B_2)$ and $Conn(a',a,A_2,B_2)$ (since we assume there is no third exchange), and $b\in Conn(a,a',A_2,B_2)$, it follows that $b\in Sym(a',A_3,B_3)$ and $b\in Sym(a,A_3',B_3')$.
\end{proof}
In order to complete the proof of Theorem~\ref{thm1} we must show that the two bases obtained as a result of the correcting exchange can be obtained from $A$ and $B$ in two consecutive symmetric exchanges. These two exchanges, along with the increasing exchange, yielded by the correcting exchange, will constitute the three desired exchanges. The next two claims are easy observations and are brought here for future reference.
\begin{claim}\label{claim:3}
Let $x,y\in A$ and $z,w\in B$ be such that $A':=(A-\{x,y\})\cup\{z,w\}$ and $B':=(B-\{z,w\})\cup\{x,y\}$ are bases. If some element of $\{x,y\}$ can be symmetrically exchanged, relative to $B'$ and $A'$, with an element of $\{z,w\}$, then $A'$ and $B'$ can be obtained from $A$ and $B$ in two consecutive symmetric exchanges.
\end{claim}
\begin{claim}\label{claim:4}
If in a correcting exchange $x=a_2$ or $x=b_2$ then a serial symmetric exchange of length three exists.
\end{claim}
\begin{proof}[Proof of Claim~\ref{claim:4}]
\renewcommand{\qedsymbol}{}
We may assume that $x=a_2$. The case $x=b_2$ follows by symmetry. Suppose that the correcting exchange $a_2$ is exchanged with $a\in A_2\cap A$. Since $b_2\in Sym(a_2,A_2,B_2)$, it follows that $b_2\in Sym(a,(A_2)_{a,a_2},(B_2)_{a_2,a})$. The result follows from Claim~\ref{claim:3}.
\end{proof}
\begin{claim}\label{claim:5}
There exists a correcting exchange such that the two bases obtained after performing it can be obtained from $A$ and $B$ in two symmetric exchanges.
\end{claim}
\begin{proof}[Proof of Claim~\ref{claim:5}]
\renewcommand{\qedsymbol}{}
This is the main part of the proof of Theorem~\ref{thm1}. The outline of the proof goes as follows. We start with the correcting exchanges $a\leftrightarrow a_i$ and $a'\leftrightarrow a_i$ from Claim~\ref{claim:2} (Figure~\ref{fig2}), and assume that the bases obtained as a result of each one of these exchanges cannot be obtained from $A$ and $B$ in two symmetric exchanges. These assumptions yield information about the relations between the elements $a_1, a_2, a, a'\in A_2$ and the elements $b_1,b_2,b\in B_2$. Next, we derive all missing the information about these relations by assuming that there is no third exchange beyond $A_2$ and $B_2$. These relations will imply the existence of at least one more element $b'\in Conn(a',a,A_2,B_2)$. We next study all possibilities for the relations between $b_1,b_2\in A_2$ and $b,b'\in B_2$ and show that each such possibility implies the existence of a correcting exchange and a pair of bases that can be obtained from $A$ and $B$ in two symmetric exchanges. This will complete the proof.

Let $A_3$, $B_3$, $A_3'$ and $B_3'$ be as in Claim~\ref{claim:2}. By Claim~\ref{claim:4} we may assume that $i=1$ in Claims~\ref{claim:1} and \ref{claim:2}. Also, by Claim~\ref{claim:3}, we may assume that the intersection of $\{b_1,b_2\}$ with each of  $Sym(a,B_3,A_3)$, $Sym(a_2,B_3,A_3)$, $Sym(a',B_3',A_3')$, and $Sym(a_2,B_3',A_3')$ is empty. Going one step back to $A_2$ and $B_2$, these assumptions imply that the relations among $a_1, a_2, b_1$ and $b_2$ must be described by one of diagrams in Figure~\ref{fig1}. We consider the cases in Figure~\ref{fig1a} and \ref{fig1b}. The other two cases will follow by duality.  Figure~\ref{fig3} illustrates these two cases (the only difference between the two diagrams is that the arrow from $a_2$ to $b_1$ is one-sided in Figure~\ref{fig1a} and two-sided in Figure~\ref{fig1b}).
\begin{figure}[h!]
  \centering
  \subfigure[]{\label{fig3a}\includegraphics[scale=0.35]{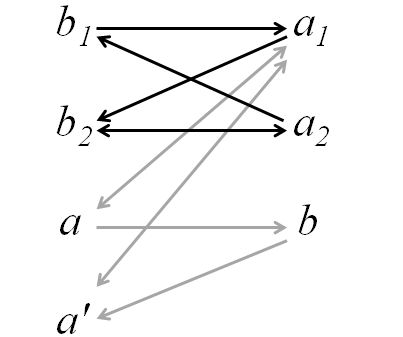}}
  \subfigure[]{\label{fig3b}\includegraphics[scale=0.35]{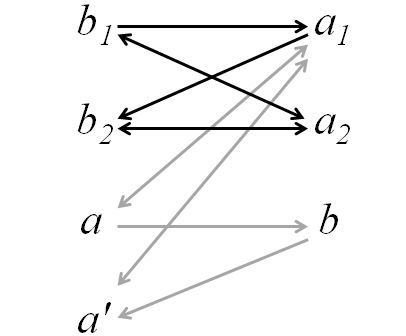}}
  \caption{The two possible scenarios given that $b_1\in Sym(a,A_2,B_2)\cap Sym(a',A_2,B_2)$ and $b_1,b_2\not\in Sym(a_1,B_2,A_2)$.}
  \label{fig3}
\end{figure}

We also assume that after the exchange $a\leftrightarrow a_1$, or the exchange $a'\leftrightarrow a_1$, the exchanges, depicted by two-sided arrows, between $b_2$ and $a_2$ (in both figures), and between $b_1$ and $a_2$ (in Figure~\ref{fig3b}), are ``ruined'' (otherwise we are done by Claim~\ref{claim:3}). We now look at the exchange $a\leftrightarrow a_1$ relative to $A_2$ and $B_2$. Since $b_2\in C(a_1,A_2)$ and $a_1\not\in C(b_2,B_2)$, we look at Proposition~\ref{prop3:4}(iii) which asserts that $Sym(b_2,A_2,B_2)\bigtriangleup Conn(b_2,a,A_2,B_2)\subseteq Sym(b_2,(A_2)_{a,a_1},(B_2)_{a_1,a})$. Since $a_2\in Sym(b_2,A_2,B_2)$ and we require that $a_2\not\in Sym(b_2,(A_2)_{a,a_1},(B_2)_{a_1,a})$, it follows that $a_2\in Conn(b_2,a,A_2,B_2)$. In particular,
\begin{equation}\label{eq3:2}
a\in C(a_2,A_2).
\end{equation}
Similarly, requiring that $a_2\not\in Sym(b_2,(A_2)_{a',a_1},(B_2)_{a_1,a'})$ implies that
\begin{equation}\label{eq3:2:1}
a'\in C(a_2,A_2).
\end{equation}
\begin{figure}[h!]
  \centering
  \subfigure[]{\label{fig4a}\includegraphics[scale=0.35]{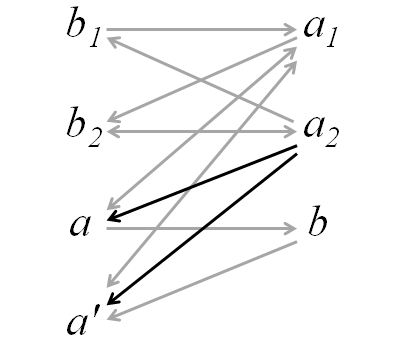}}
  \subfigure[]{\label{fig4b}\includegraphics[scale=0.35]{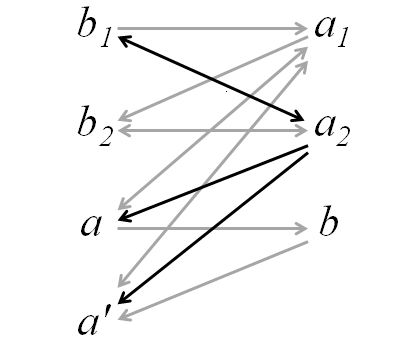}}
  \caption{Requiring $a_2\not\in Sym(b_2,(A_2)_{a,a_1},(B_2)_{a_1,a})$ and $a_2\not\in Sym(b_2,(A_2)_{a',a_1},(B_2)_{a_1,a'})$ implies that $a,a'\in C(a_2,A_2)$ }
  \label{fig4}
\end{figure}

We now assume that (\ref{eq3:2}) and (\ref{eq3:2:1}) hold (Figure~\ref{fig4}) and look at the other ends of the two new arrows representing these relations. Suppose $a_2\in C(a',B_2)$. Together with (\ref{eq3:2:1}) we have an exchange $a'\leftrightarrow a_2$, relative to $A_2$ and $B_2$. If $a_2\not\in C(a,B_2)$ we apply Proposition~\ref{prop3:4}(iii) which asserts that $Sym(a,A_2,B_2)\bigtriangleup Conn(a,a',A_2,B_2)\subseteq Sym(a,(A_2)_{a',a_2},(B_2)_{a_2,a'})$ (note that the roles of $a$ and $a'$ are interchanged with respect to their roles in Proposition~\ref{prop3:4}). Since $b\in Conn(a,a',A_2,B_2)$ but $b\not\in Sym(a,A_2,B_2)$ (assuming there is no third exchange), it follows that $b\in Sym(a,(A_2)_{a',a_2},(B_2)_{a_2,a'})$. If $a_2\in C(a,B_2)$, we apply Proposition~\ref{prop3:4}(iv) which asserts that $Sym(a,A_2,B_2)\bigtriangleup Sym(a',A_2,B_2)\bigtriangleup Conn(a,a',A_2,B_2)\bigtriangleup Conn(a',a,A_2,B_2)+a'\subseteq Sym(a,(A_2)_{a',a_2},(B_2)_{a_2,a'})$. Since $b\in Conn(a,a',A_2,B_2)$ but $b\not\in Conn(a',a,A_2,B_2)$, $b\not\in Sym(a,A_2,B_2)$ and $b\not\in Sym(a',A_2,B_2)$ (assuming there is no third exchange) we have again that $b\in Sym(a,(A_2)_{a',a_2},(B_2)_{a_2,a'})$. In any case we obtain that if $a_2\in C(a',B_2)$ then $b\in Sym(a,(A_2)_{a',a_2},(B_2)_{a_2,a'})$. Thus, the exchange $a'\leftrightarrow a_2$ is a correcting exchange, and since $a_2$ is involved, we are done by Claim~\ref{claim:4}. Hence, we now assume that
\begin{equation}\label{eq3:3}
    a_2\not\in C(a',B_2).
\end{equation}

We consider the case in Figure~\ref{fig4b} and we require that $a_2\not\in Sym(b_1,(A_2)_{a',a_1},(B_2)_{a_1,a'})$ (otherwise we are done, by Claim~\ref{claim:3}). By Proposition~\ref{prop3:4}(ii), $Sym(b_1,A_2,B_2)\bigtriangleup Conn(a',b_1,A_2,B_2)\subseteq Sym(b_1,(A_2)_{a',a_1},(B_2)_{a_1,a'})$. Since $a_2\in Sym(b_1,A_2,B_2)$ and we require that $a_2\not\in Sym(b_1,(A_2)_{a',a_1},(B_2)_{a_1,a'})$, we must have that $a_2\in Conn(a',b_1,A_2,B_2)$. In particular, $a_2\in C(a',B_2)$. This contradicts (\ref{eq3:3}). Thus, we may drop the case in Figure~\ref{fig4b} and assume that the arrow between $a_2$ and $b_1$ is one-sided.

Next, $Conn(a',a,A_2,B_2)$ contains $a_1$ and does not contain $a_2$ and $b$ (this is easy to see in Figure~\ref{fig4a}). By Proposition~\ref{prop:KZ}, there must be some $b'\in B_2\cap B$ such that $b'\in Conn(a',a,A_2,B_2)$ (Figure~\ref{fig5}). We now consider the other end of the arrow between $a_2$ and $a$. If $a_2\in C(a,B_2)$, then, by (\ref{eq3:2}), $a_2\in Sym(a,A_2,B_2)$. We exchange $a$ with $a_2$ and consider $a'$. Since $a'\in C(a_2,A_2)$ and $a_2\not\in C(a',B_2)$ we look at Proposition~\ref{prop3:4}(iii) which asserts $Sym(a',A_2,B_2)\bigtriangleup Conn(a',a,A_2,B_2)\subseteq Sym(a',(A_2)_{a,a_2},(B_2)_{a_2,a})$. Since $b'\not\in Sym(a',A_2,B_2)$ and $b'\in Conn(a',a,A_2,B_2)$, we have $b'\in Sym(a',(A_2)_{a,a_2},(B_2)_{a_2,a})$. Thus, we have a correcting exchange involving $a_2$, and we are done by Claim~\ref{claim:4}. Hence, we may assume that $a_2\not\in C(a,B_2)$. Figure~\ref{fig5} represents all the relations obtained so far, between elements of $A_2$ and $B_2$, assuming there is no third exchange.
\begin{figure}[h!]
\begin{center}
\includegraphics[scale=0.35]{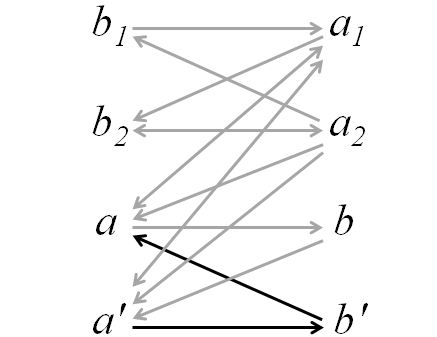}
\end{center}
\caption{All the relations between elements of $A_2$ and $B_2$ assuming there is no third exchange.}
\label{fig5}
\end{figure}

We still have no information about the relations between $\{b_1,b_2\}$ and $\{b,b'\}$. Assuming there is no third exchange, each of $b$ and $b'$ must be exchangeable, relative to $A_2$ and $B_2$, with at least one of $b_1$ and $b_2$. There are three options (ranked by complexity of verification):
\begin{enumerate}
  \item both $b$ and $b'$ are exchangeable with $b_2$,
  \item both $b$ and $b'$ are exchangeable with $b_1$, and
  \item each of $b$ and $b'$ is exchangeable with a different $b_i$ ($i=1,2$).
\end{enumerate}

It will be shown that in each of these cases there exists a serial symmetric exchange of length three, by finding a correcting exchange that yields two bases that can be obtained from $A$ and $B$ in two symmetric exchanges.

(1) Suppose $b$ and $b'$ are both exchangeable with $b_2$. Since $a'\in Conn(b,b',B_2,A_2)$, the situation is symmetric to the one described in Figure~\ref{fig2}, with $b_2$ assuming the role of $a_i$, $b$ and $b'$ assuming the roles of $a$ and $a'$ respectively, and $a'$ assuming the role of $b$ (Figure~\ref{fig6a}). Thus, $b_2\leftrightarrow b$ and $b_2\leftrightarrow b'$ are correcting exchanges, and we are done by Claim~\ref{claim:4}.
\begin{figure}[h!]
  \centering
  \subfigure[]{\label{fig6a}\includegraphics[scale=0.35]{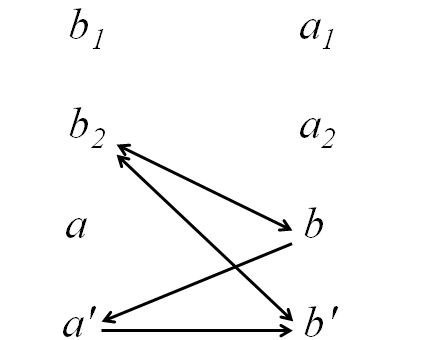}}
  \subfigure[]{\label{fig6b}\includegraphics[scale=0.35]{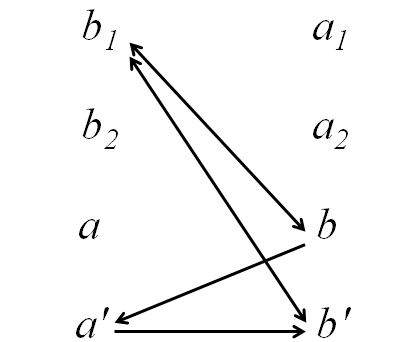}}
  \caption{Cases (1) and (2) are symmetric to the case described in Figure~\ref{fig2}. }
  \label{fig6}
\end{figure}

(2) Suppose $b$ and $b'$ are both exchangeable with $b_1$. Again, the situation is symmetric to the one described in Figure~\ref{fig2}, with $b_1$ assuming the role of $a_i$, $b$ and $b'$ assuming the roles of $a$ and $a'$ respectively, and $a'$ assuming the role of $b$ (Figure~\ref{fig6b}). Thus $b_1\leftrightarrow b$ and $b_1\leftrightarrow b'$ are correcting exchanges. It is left to the reader to verify that in order to ``ruin'' the exchange $b_2\leftrightarrow a_2$ as a result of either one of the exchanges $b_1\leftrightarrow b$ and $b_1\leftrightarrow b'$, we must require the one-sided relations $b_2\leftarrow b$ and $b_2\leftarrow b'$ (Figure~\ref{fig7a}) (the procedure is very similar to the one used above to obtain the one sided relations $a_2\rightarrow a$ and $a_2\rightarrow a'$ in Figure~\ref{fig4}). Now, $a_2\in Conn(b_2,a,A_2,B_2)$ and $a_1,b,b'\not\in Conn(b_2,a,A_2,B_2)$ (this is easy to see in Figure~\ref{fig4}). Thus, by Proposition~\ref{prop:KZ}, there exists an element $b''\in Conn(b_2,a,A_2,B_2)\cap B$ (Figure~\ref{fig7b}).
\begin{figure}[h!]
  \centering
  \subfigure[]{\label{fig7a}\includegraphics[scale=0.35]{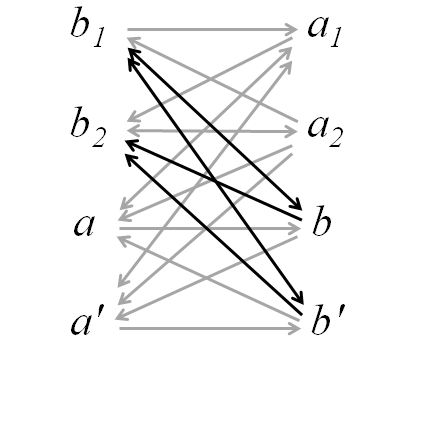}}
  \subfigure[]{\label{fig7b}\includegraphics[scale=0.35]{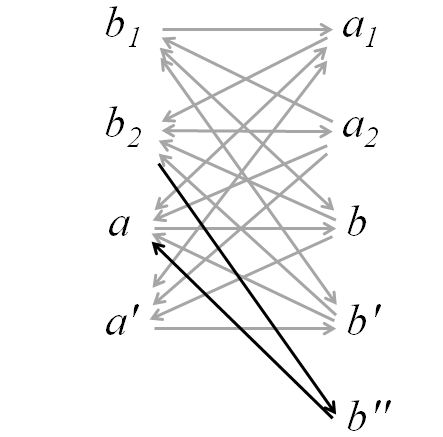}}
  \caption{The relations implied in Case (2). }
  \label{fig7}
\end{figure}

Assuming there is no third exchange, $b''$ must be symmetrically exchangeable with at least one of $b_1$ and $b_2$. Suppose $b_2\in Sym(b'',B_2,A_2)$. If we perform the exchange $b_2\leftrightarrow b''$, then the element $a$ corresponds to Proposition~\ref{prop3:4}(iii), which asserts that $Sym(a,A_2,B_2)\bigtriangleup Conn(a,b_2,A_2,B_2)\subseteq Sym(a,(A_2)_{b_2,b''},(B_2)_{b'',b_2})$. Since $b\in Conn(a,b_2,A_2,B_2)\setminus Sym(a,A_2,B_2)$ it follows that $b\in Sym(a,(A_2)_{b_2,b''},(B_2)_{b'',b_2})$. Thus $b_2\leftrightarrow b''$ is a correcting exchange and we are done by Claim~\ref{claim:4}. Hence we assume that the arrow $b_2\rightarrow b''$ is one-sided and $b_1\in Sym(b'',B_2,A_2)$. If we perform the exchange $b_1\leftrightarrow b''$, then the element $a$ corresponds to Proposition~\ref{prop3:4}(iii). Thus, we have $Sym(a,A_2,B_2)\bigtriangleup Conn(a,b_1,A_2,B_2)\subseteq Sym(a,(A_2)_{b_1,b''},(B_2)_{b'',b_1})$. Since $b\in Conn(a,b_1,A_2,B_2)\setminus Sym(a,A_2,B_2)$ it follows that $b\in Sym(a,(A_2)_{b_1,b''},(B_2)_{b'',b_1})$, so that $b_1\leftrightarrow b''$ is a correcting exchange. It is claimed that $(A_2)_{b_1,b''}$ and $(B_2)_{b'',b_1}$ can be obtained from $A$ and $B$ in two symmetric exchanges. Note that with respect to the exchange $b_1\leftrightarrow b''$ the element $b_2$ corresponds to Case (ii) of Proposition~\ref{prop3:4}. Thus $Sym(b_2,A_2,B_2)\bigtriangleup Conn(b_1,b_2,A_2,B_2)\subseteq Sym(b_2,(A_2)_{b_1,b''},(B_2)_{b'',b_1})$. Since $a_2\in Sym(b_2,A_2,B_2)\setminus Conn(b_1,b_2,A_2,B_2)$ it follows that $a_2\in Sym(b_2,(A_2)_{b_1,b''},(B_2)_{b'',b_1})$. Thus, by Claim~\ref{claim:3}$, (A_2)_{b_1,b''}$ and $(B_2)_{b'',b_1}$ can be obtained from $A$ and $B$ in two symmetric exchanges, and we are done.

(3) Suppose each of $b$ and $b'$ is exchangeable, relative to $A_2$ and $B_2$, with a different $b_i$ ($i=1,2$). We assume that $b$ is exchangeable with $b_1$ and $b'$ is exchangeable with $b_2$ (the case that $b$ is exchangeable with $b_2$ and $b'$ is exchangeable with $b_1$ can be treated using a similar technique and is left to the reader). There are three possibilities for the relation between $b_2$ and $b$ (Figure~\ref{fig8}):
\begin{enumerate}
  \item[(i)] $b_2\in C(b, A_2)$ and $b\not\in C(b_2,B_2)$,
  \item[(ii)] $b_2\not \in C(b, A_2)$ and $b\in C(b_2,B_2)$, and
  \item[(iii)] no relation at all
\end{enumerate}
(the case that both $b_2\in C(b, A_2)$ and $b\in C(b_2,B_2)$ has already been considered in Case (1) above).
In all three cases we will show that the exchange $b_2\leftrightarrow b'$ is a correcting exchange and the result will follow from Claim~\ref{claim:4}.

(i) Suppose $b_2\in C(b, A_2)$ (Figure~\ref{fig8a}). Since $a\in C(b',A_2)$ and $b'\not\in C(a,B_2)$, Proposition~\ref{prop3:4}(iii) implies that $Sym(a,A_2,B_2)\bigtriangleup Conn(a,b_2,A_2,B_2)\subseteq Sym(a,(A_2)_{b_2,b'},(B_2)_{b',b_2})$. Since $b\in Conn(a,b_2,A_2,B_2)\setminus Sym(a,A_2,B_2)$ it follows that $b\in Sym(a,(A_2)_{b_2,b'},(B_2)_{b',b_2})$. Thus, $a\leftrightarrow b$ is a correcting exchange.

\begin{figure}[h!]
  \centering
  \subfigure[Case (3)(i)]{\label{fig8a}\includegraphics[scale=0.35]{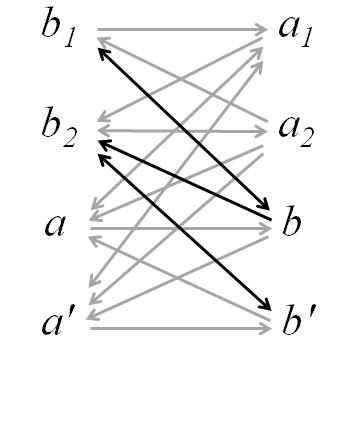}}
  \subfigure[Case (3)(ii)]{\label{fig8b}\includegraphics[scale=0.35]{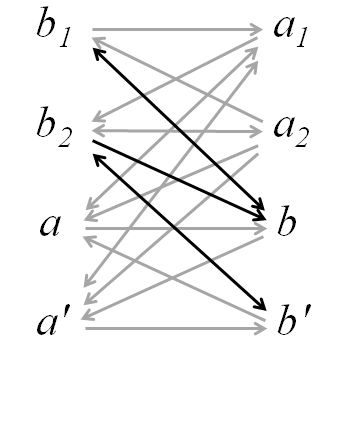}}
  \subfigure[Case (3)(iii)]{\label{fig8c}\includegraphics[scale=0.35]{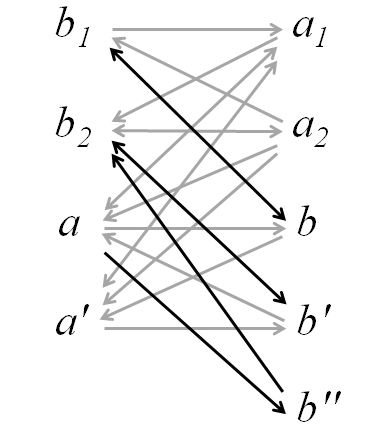}}
  \caption{The three possible scenarios in Case (3), depending on the relation between $b_2$ and $b$.}
  \label{fig8}
\end{figure}

(ii) Suppose $b\in C(b_2,B_2)$ (Figure~\ref{fig8b}). Since $a'\not\in C(b',A_2)$ and $b'\in C(a',B_2)$, Proposition~\ref{prop3:4}(ii) indicates that $Sym(a',A_2,B_2)\bigtriangleup Conn(b_2,a',A_2,B_2)\subseteq Sym(a',(A_2)_{b_2,b'},(B_2)_{b',b_2})$. Since $b\in Conn(b_2,a',A_2,B_2)\setminus Sym(a',A_2,B_2)$ it follows that $b\in Sym(a',(A_2)_{b_2,b'},(B_2)_{b',b_2})$. Thus, $a'\leftrightarrow b$ is a correcting exchange.

(iii) If there is no relation between $b_2$ and $b$ (Figure~\ref{fig8c}) we look at $Conn(a,b_2,A_2,B_2)$ which contains only $a_1$ of the four elements considered so far in $B_2$. Thus there exists $b''\in Conn(a,b_2,A_2,B_2)\cap B$ (Figure~\ref{fig8c}). The rest follows as in Case (i), with $b''$ in place $b$.

It was shown that all three cases (1)-(3) yield a third exchange. This completes the proof of Claim~\ref{claim:5} and the proof of Theorem~\ref{thm1}.
\end{proof} 
\end{proof} 
\begin{rem}
The proof of Theorem~\ref{thm1} starts with a pair of elements, $a_1,a_2\in A$ that are serially and symmetrically exchangeable, relative to $A$ and $B$, with a pair $b_1,b_2\in B$. Since, by \cite{KZ2011}, each pair of elements in $A$ is part of such an exchange, we can expect, as $r(M)$ becomes larger, to find a growing amount of distinct triples $a_1,a_2,a_3\in A$ that are part of a serial symmetric exchange of length 3.
\end{rem}
\section{The case $r(M)=5$}\label{sec4}
\begin{thm}
Any two bases of a matroid of rank 5 have a full serial symmetric exchange of length at most 6.
\end{thm}
\begin{proof}
Let $M$ be a matroid of rank 5, and let $A=\{a_1,\ldots,a_5\}$ and $B=\{b_1,\ldots,b_5\}$ be two bases in $M$. By Theorem~\ref{thm1} we may assume that $\{a_1,a_2,a_3\}$ can be serially and symmetrically exchanged with $\{b_1,b_2,b_3\}$. Let $A'=(A\setminus \{a_1,a_2,a_3\})\cup \{b_1,b_2,b_3\}$ and $B'=(B\setminus \{b_1,b_2,b_3\})\cup \{a_1,a_2,a_3\}$. If there is still an exchange, relative to $A'$ and $B'$, between one of $a_4$ or $a_5$ and $b_4$ or $b_5$, then a full serial symmetric exchange of length 5 exists. So, we assume that there is no fourth exchange. Since each of $C(a_4,B')$ and $C(a_5,B')$ must contain an element of $B$ and each of $C(b_4,A')$ and $C(b_5,A')$ must contain an element of $A$, and given that there is no fourth exchange, we may assume that the relations among $a_4,a_5,b_4$ and $b_5$, relative to $A'$ and $B'$, are as illustrated in Figure~\ref{fig9}.
\begin{figure}[h!]
\begin{center}
\includegraphics[scale=0.35]{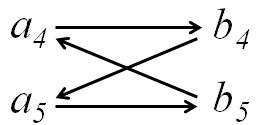}
\end{center}
\caption{The relations among $a_4,a_5,b_4$ and $b_5$, given that there is no fourth exchange.}
\label{fig9}
\end{figure}

Now, $b_5\in Conn(a_5,a_4,A',B')$, so at least one element of $\{a_1,a_2,a_3\}$ must be in $Conn(a_5,a_4,A',B')$, by Proposition~\ref{prop:KZ}. Similarly, at least one element of $\{a_1,a_2,a_3\}$ must be in $Conn(a_4,a_5,A',B')$. In addition, each of $Sym(a_4,A',B')$ and $Sym(a_5,A',B')$ must contain an element of $\{a_1,a_2,a_3\}$, since we assume there is no fourth exchange. All these imply that for some $i\in\{1,2,3\}$, and $j\in\{4,5\}$, $a_j$ is exchangeable, relative to $A'$ and $B'$, with $a_i$, and there is a relation between $a_k$ ($k\in\{4,5\},k\ne j$) and $a_i$ (either $a_k\in C(a_i,A')$ or $a_i\in C(a_k,B')$ or both). We assume $i=1$, and $j=4$, and leave it to the reader to verify the other cases, using similar methods. We perform the exchange $a_4\leftrightarrow a_1$ (this is a "correcting exchange", in the terminology used in the proof of Theorem~\ref{thm1}). If $a_1\in C(a_5,B')$ and $a_5\not\in C(a_1,A')$ (Figure~\ref{fig10a}) then, by Proposition~\ref{prop3:4}(ii), we have $Sym(a_5,A',B')\bigtriangleup Conn(a_4,a_5,A',B')\subseteq Symm(a_5,A'_{a_4,a_1},B'_{a_1,a_4})$. Since $b_4\not\in Sym(a_5,A',B')$ and $b_4\in Conn(a_4,a_5,A',B')$ it follows that $b_4\in Symm(a_5,A'_{a_4,a_1},B'_{a_1,a_4})$ (Figure~\ref{fig10b}). After exchanging $a_5\leftrightarrow b_4$, the sixth and last exchange is $a_1\leftrightarrow b_5$ (Figure~\ref{fig10c}). If $a_5\in C(a_1,A')$ and $a_1\not\in C(a_5,B')$ then, by Proposition~\ref{prop3:4}(iii), we have $Sym(a_5,A',B')\bigtriangleup Conn(a_5,a_4,A',B')\subseteq Symm(a_5,A'_{a_4,a_1},B'_{a_1,a_4})$. Since $b_5\in Conn(a_5,a_4,A',B')\setminus Sym(a_5,A',B')$ it follows that $b_5\in Symm(a_5,A'_{a_4,a_1},B'_{a_1,a_4})$. Exchanging $a_5\leftrightarrow b_5$ and then $a_1\leftrightarrow b_4$ completes the full exchange. If both $a_1\in C(a_5,B')$ and $a_5\in C(a_1,A')$ hold, then by Proposition~\ref{prop3:4}(iv), we have $Sym(a_4,A',B')\bigtriangleup Sym(a_5,A',B')\bigtriangleup Conn(a_4,a_5,A',B')\bigtriangleup Conn(a_5,a_4,A',B')\subseteq Symm(a_5,A'_{a_4,a_1},B'_{a_1,a_4})$. Since $b_5$ is only in one of the terms on the left hand side ($Conn(a_5,a_4,A',B')$) it must be in $Symm(a_5,A'_{a_4,a_1},B'_{a_1,a_4})$. Thus,  exchanging $a_5\leftrightarrow b_5$ and then $a_1\leftrightarrow b_4$ will complete the full serial symmetric exchange.
\begin{figure}[h!]
  \centering
  \subfigure[]{\label{fig10a}\includegraphics[scale=0.35]{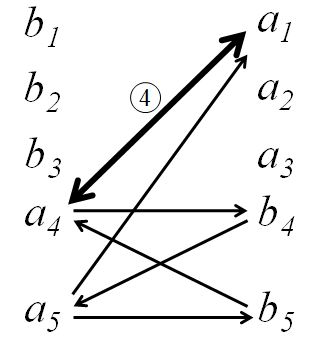}}
  \subfigure[]{\label{fig10b}\includegraphics[scale=0.35]{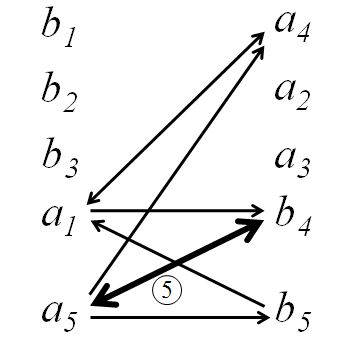}}
  \subfigure[]{\label{fig10c}\includegraphics[scale=0.35]{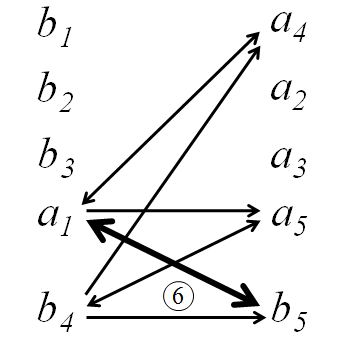}}
  \caption{ The last three symmetric exchanges when $r(M)=5$.}
  \label{fig10}
\end{figure}
\end{proof}
\section{The case of a binary matroid}\label{sec5}

When $M$ is binary the first inclusion in Proposition~\ref{prop2:1} is an equality:
\begin{prop}\label{prop6:1}
Let $B$ be a basis in a binary matroid $M$. If $x,y\not\in B$, $x\ne y$, $b\in C(x,B)\cap C(y,B)$ and $B_{b,x}=B-b+x$, then
\begin{equation}\label{eq2}
C(y,B_{b,x})=C(x,B)\bigtriangleup C(y,B).
\end{equation}
\end{prop}
\begin{proof}
In view of Proposition~\ref{prop2:1} it remains to show that if $b'\in C(x,B)\cap C(y,B)$, then $b'\not\in C(y,B_{b,x})$. The hypotheses imply that $\cl(B-b)$, $\cl(B\setminus\{b,b'\}+ x)$, and $\cl(B-b')$ are distinct hyperplanes that contain $\cl(B\setminus\{b,b'\})$. Since $M$ is binary, they are the only such hyperplanes. Also, $y\not\in\cl(B-b)$ and $y\not\in\cl(B-b')$. Thus, $y$ must be in $\cl(B\setminus\{b,b'\}+ x)$, which implies that $b'\not\in C(y,B_{b,x})$.
\end{proof}
\begin{rem}
Proposition~\ref{prop6:1} implies that for a binary matroid all the first inclusions in parts (ii)-(iv) of Proposition~\ref{prop3:4} are equalities.
\end{rem}
The property of Proposition~\ref{prop6:1} characterizes binary matroids:
\begin{thm}
The following are equivalent for a matroid $M$:
\begin{enumerate}
\item[\Text{(i)}] $M$ is binary.
\item[\Text{(ii)}] For any basis $B$ of $M$, whenever an element $b\in B$ is replaced by $x\not\in B$ to obtain a basis $B_{b,x}=B-b+x$, for any $y\not\in B$, $y\ne x$, such that $b\in C(y,B)$, $C(y,B_{b,x}) =C(x,B)\bigtriangleup C(y,B)$.
\end{enumerate}
\end{thm}\label{thm5:3}
\begin{proof}
It remains to show that (ii) implies (i). Assume $M$ is not binary, then $M$ has a $U_{2,4}$-minor, so some flat $F$ of rank
$r(M)-2$ is contained in four distinct hyperplanes $H_1,H_2,H_3,H_4$ (perhaps more) of $M$. Let $B_0$ be a basis of $F$. For $i\in \{1,2,3,4\}$, pick $a_i\in H_i-F$. Note that $B_0\cup \{a_i,a_j\}$ is a basis of $M$ whenever $\{i,j\}\subset \{1,2,3,4\}$. Set $B=B_0\cup\{a_1,a_2\}$. Let $B_{a_1,a_3} = B_0\cup\{a_2,a_3\}$ (see Notation~\ref{notn1}). Note that $a_2$ is in each of $C(a_4,B)$, $C(a_4,B_{a_1,a_3})$, and $C(a_3,B)$. Thus $C(a_4,B_{a_1,a_3})\ne C(a_4,B)\triangle C(a_3,B)$, although $a_1\in C(a_4,B)$.
\end{proof}

\begin{rem}
Rota conjectured and Greene \cite{Greene71} proved that a matroid $M$ is binary if and only if $|Sym(a,A,B)|$ is odd for any two disjoint bases $A$ and $B$ and any $a\in A$. From Proposition~\ref{prop3:1}(ii) it follows that if $M$ is binary, then $|Conn(a,a',A,B)|$ is even for any two disjoint bases $A$ and $B$ and any $a,a'\in A,\;a\ne a'$. However, this is not a characterization of binary matroids, since $U_{2,4}$ is a counter-example.
\end{rem}
\bibliographystyle{abbrv}
\bibliography{base_ex}
\end{document}